\documentclass[11pt]{amsart}
\usepackage{graphicx}
\usepackage[centertags]{amsmath}
\usepackage{amsfonts}
\usepackage{amssymb}
\usepackage{amsthm}
\usepackage{newlfont}
\usepackage[T2A]{fontenc}
\usepackage[utf8]{inputenc}
\usepackage[english]{babel}
\usepackage{float}
\newtheorem{theorem}{Theorem}[section]
\newtheorem{corollary}[theorem]{Corollary}

\newtheorem{prop}[theorem]{Proposition}
\theoremstyle{definition}
\newtheorem{defn}[theorem]{Definition}
\theoremstyle{remark}
\newtheorem{rem}[theorem]{Remark}

\newcommand{\id}{\textrm{id}}

\newcommand{\di}{\textrm{d}}
\title[Four-dimensional almost-Einstein manifolds]{Four-dimensional almost-Einstein manifolds with circulant structures}

\author{Iva Dokuzova, Georgi Dzhelepov, Dimitar Razpopov}

\address{Iva Dokuzova\\ Department of Algebra and Geometry \\ Faculty of Mathematics and Informatics\\ University of Plovdiv Paisii Hilendarski\\ 236 Bulgaria Blvd., 4027 Plovdiv, Bulgaria}
\email{dokuzova@uni-plovdiv.bg}
\address{Georgi Dzhelepov\\ Department of Mathematics and Informatics\\ Faculty of Economics\\ Agricultural University of Plovdiv\\ 12 Mendeleev Blvd., 4000 Plovdiv, Bulgaria}
\email{dzhelepov@abv.bg}
\address{Dimitar Razpopov\\ Department of Mathematics and Informatics\\ Faculty of Economics\\ Agricultural University of Plovdiv\\ 12 Mendeleev Blvd., 4000 Plovdiv, Bulgaria}
\email{razpopov@au-plovdiv.bg}

\begin{document}

%
\begin{abstract}
This paper considers a four-dimensional Riemannian manifold with an additional circulant tensor structure $Q$. We prove that such a manifold is almost-Einstein if the curvature tensor preserves the action of the structure on the last two arguments. 
Additionally, we examine an associated Riemannian product manifold $(M, g, Q)$ where $P=Q^{2}$, deriving explicit expressions for the curvature tensor in terms of the metric and the product structure. The sectional and Ricci curvatures of these manifolds are determined using special bases generated by the images of a given vector under the structure $Q$. 
Finally, we construct a non-trivial example of a four-dimensional locally product manifold by embedding a four-dimensional hypersurface into a ten-dimensional Euclidean space.
\end{abstract}



\maketitle

\section{Introduction}
The foundations of the study of Riemannian manifolds with an almost-product structure $P$ were laid by K.~Yano in \cite{yano}. Subsequently, the classification of Riemannian almost-product manifolds with respect to the covariant derivative of $P$ was obtained by A.~M.~Naveira in \cite{nav}, while
such manifolds with a zero trace of $P$ were classified by M.~Staikova and K.~Gribachev in \cite{sta}.

The most restricted class within these framework is the class of locally product Riemannian manifolds (or Riemannian $P$-manifolds), which consists of manifolds with a parallel structure $P$. The metric tensor of such a manifold can be decomposed into two independent metric tensors, making the manifold a product of two submanifolds. A more general class within this classification is the class of Riemannian product manifolds, for which the Nijenhuis tensor vanishes.  The curvature properties of the Riemannian product manifolds have been studied, for instance, in \cite{asmari, esra, gribacheva2, griba-mek, jack, pusic, sta2}. Furthermore, it was proved by D. Gribacheva and D. Mekerov in \cite{griba-mek2} that every four-dimensional Riemannian almost-product manifold whose curvature tensor is a Riemannian $P$-tensor is an almost-Einstein manifold, with its curvature tensor expressed in terms of the metric and the almost-product structure. 

A four-dimensional differentiable manifold $(M, g, Q)$ equipped with an additional circulant structure $Q$ and a Riemannian metric $g$ is defined in \cite{raz}. The geometric properties of such a manifold are investigated in \cite{dokuzova}, \cite{dokuzova1}, \cite{dokuzova2}, and \cite{razpopov2}. In the present paper, we consider a Riemannian manifold $(M, g, Q)$ whose curvature tensor preserves the action of $Q$ on the last two arguments. We show that such a manifold is necessarily almost-Einstein. Furthermore, we find an explicit expression for the curvature tensor in terms of the metric $g$ and the structure $Q$. We obtain the Ricci curvatures in the directions of the basis vectors determined by $Q$-bases in the tangent spaces of the manifold.  We also express the sectional curvatures of special 2-planes determined by $Q$-bases in terms of the scalar curvatures and the angles between the basis vectors.
Additionally, we study the curvature properties of the associated Riemannian product manifold $(M, g, P)$, where $P=Q^{2}$. Such manifolds form a class of almost-Einstein manifolds when their curvature tensors are Riemannian $P$-tensors (i.e., the action of $P$ on the last two arguments is preserved).
In the final section, we construct a non-trivial example of a locally product Riemannian manifold $(M, g, P)$. 

\section{Preliminaries}\label{sec:1}
The four-dimensional manifolds $(M, g, Q)$ and $(M, g, P)$ were introduced in \cite{raz} and \cite{dokuzova}.  In this section, we provide the basic facts regarding them.

 Let $M$ be a $4$-dimensional Riemannian manifold equipped with a tensor structure $Q$ at each point $p$ on $M$. Assume that $Q$ has the following component matrix:
\begin{equation}\label{Q}
    (Q_{j}^{k})=\begin{pmatrix}
      0 & 1 & 0 & 0\\
      0 & 0 & 1 & 0\\
      0 & 0 & 0 & 1\\
      1 & 0 & 0 & 0\\
    \end{pmatrix}.
\end{equation}
Then, $Q$ obey relations:
\begin{equation}\label{Q4}
    Q^{4}=\id,\quad Q^{2}\neq\pm\id.
\end{equation}
We assume that the metric $g$ and the structure $Q$ satisfy equality
\begin{equation}\label{isometry}
 g(Qx, Qy)=g(x,y),\quad x, y\in \mathfrak{X}(M).
\end{equation}
Throughout this work, $x, y, z, v$ will denote arbitrary elements of the algebra of the smooth vector fields on $M$ or vectors in the tangent space $T_{p}M$. The Einstein summation convention is used, with the range of the summation indices being always $\{1, 2, 3, 4\}$.

Due to \eqref{Q4} and \eqref{isometry}, the structure $P=Q^{2}$ is an almost-product structure and satisfies
\begin{equation*}
    g(Px, Py)=g(x, y).
\end{equation*}
Bearing in mind conditions \eqref{Q} and \eqref{isometry}, the structures $P$ and $g$ have local coordinates given by the following circulant matrices:
\begin{equation}\label{S}
    (P_{i}^{j})=\begin{pmatrix}
      0 & 0 & 1 & 0\\
      0 & 0 &  0 & 1 \\
      1 & 0 & 0 & 0\\
      0 & 1 & 0 & 0\\
    \end{pmatrix},
\end{equation}
\begin{equation}\label{metricg}
    (g_{ij})=\begin{pmatrix}
      A & B & C & B \\
      B & A & B & C\\
      C & B & A & B\\
      B & C & B & A\\
    \end{pmatrix}.
\end{equation}
We assume that $A$, $B$, and $C$ are smooth functions of an arbitrary point $p(x^{1}, x^{2},
x^{3}, x^{4})$ on $M$, satisfying $A >C> B>0$ to ensure that $g$ is a positive definite metric.

We consider the tensor $\tilde{g}(x,y)=g(Px,y)$. Its component matrix takes the form
\begin{equation}\label{metricJ}
    (\tilde{g}_{ij})=\begin{pmatrix}
      C & B & A & B \\
      B & C & B & A\\
      A & B & C & B\\
      B & A & B & C\\
    \end{pmatrix}
\end{equation}
and is symmetric, i.e., $\tilde{g}(x,y)=\tilde{g}(y,x)$.

 The associated manifold $(M, g, P)$ is a Riemannian product manifold, and the tensor $\tilde{g}$ is an indefinite metric under the conditions $A >C> B>0$.

We now present well-known curvature objects on a Riemannian manifold. 
The curvature tensor $R$ is defined by $$R(x, y)z=\nabla_{x}\nabla_{y}z-\nabla_{y}\nabla_{x}z-\nabla_{[x,y]}z,$$ where $\nabla$ is the Levi-Civita connection of $g$.
The tensor of type $(0, 4)$ associated with $R$ is given by
\begin{equation*}
    R(x, y, z, u)=g(R(x, y)z,u).
\end{equation*}

The Ricci tensor $\rho$ and the scalar curvature $\tau$ with respect to $g$ are defined by:
\begin{equation}\label{def-rho}
    \rho(y,z)=g^{ij}R(e_{i}, y, z, e_{j}),\
    \tau=g^{ij}\rho(e_{i}, e_{j}).
\end{equation}

The Ricci curvature with respect to $g$ in the direction of a nonzero vector $x$ is given by \begin{equation}\label{Ricicurv}
    r(x)=\frac{\rho(x,x)}{g(x,x)}.
\end{equation}

The curvature of a non-degenerate $2$-plane $\{x, y\}$  is
\begin{equation}\label{3.3}
    \kappa(x,y)=\frac{R(x, y, x, y)}{g(x, x)g(y, y)-g^{2}(x, y)}.
\end{equation}

For Riemannian almost-product manifolds, Naveira's classification (\cite{nav}) and Staikova-Gribachev's classification (\cite{sta}) are valid. The class  $W_{0}$ of so-called locally product Riemannian manifolds is defined by
\begin{equation}\label{keler} \nabla P=0.\end{equation}
This class is common to all classes within these classifications.

Now, we recall the conditions for the parallelism of the structures $Q$ and $P$, obtained in \cite{raz} and \cite{dokuzova}, respectively.
\begin{itemize}
\item The structure $Q$ on $(M, g, Q)$ is parallel with respect to $\nabla$ if and only if the following conditions hold:
\begin{equation}\label{paralel-f4}
\begin{split}
A_{3}&=C_{1},\ A_{1}=C_{3},\ A_{2}=C_{4},\ A_{4}=C_{2},\\
2B_{4}&=2B_{2}=C_{1}+C_{3},\ 2B_{3}=2B_{1}=C_{2}+C_{4}.
\end{split}
\end{equation}
Here $A_{i}$, $B_{i}$, and $C_{i}$ are partial derivatives of the functions $A$, $B$, and $C$.
\item The structure $P=Q^{2}$ on $(M, g, P)$ is parallel with respect to $\nabla$ if and only if 
\begin{equation}\label{w0}
A_{3}=C_{1},\ A_{1}=C_{3},\  A_{2}=C_{4},\ B_{3}=B_{1},\ A_{4}=C_{2},\
B_{4}=B_{2}.
\end{equation}
\end{itemize}
\begin{rem} We note that if $\nabla Q=0$, then \eqref{keler} holds, but the converse is not true.
\end{rem}

It is known that a Riemannian manifold $(M, g)$ is called almost-Einstein if its Ricci tensor $\rho$ is represented as a linear combination of the metrics $g$ and $\tilde{g}$ with smooth functions on $M$ as coefficients, i.e.,
\begin{equation}\label{AE}\rho(x, y) = \alpha g(x, y)+\beta \tilde{g}(x, y).\end{equation}

\section{Almost-Einstein manifolds $(M, g, P)$}
In this section, we consider a manifold $(M, g, P)$ whose curvature tensor $R$ is a Riemannian $P$-tensor, meaning that the action of $P$ on the last two arguments is preserved. 

For the curvature tensor and the structure $P$, condition \eqref{keler} implies:
\begin{equation}\label{3.2-P} R(x,y,Pz,Pu)=R(x,y,z,u).\end{equation}
This identity defines a class of manifolds $(M, g, P)$ that is more general than the class of manifolds satisfying \eqref{keler}. Condition \eqref{3.2-P} determines $R$ as a Riemannian $P$-tensor \cite{sta2}.

We establish the following results.
\begin{prop}\label{svoistvo-R-P}
Identity \eqref{3.2-P} for the curvature tensor $R$ on the manifold $(M, g, P)$ is equivalent to the following conditions for its local components:
\begin{equation}\label{R-loc-P}
\begin{array}{ll}
R_{1212}&=R_{3434}=R_{2323}=R_{1414}=R_{1234}=R_{4123},\\
R_{2114}&=R_{4134}=R_{1223}=R_{2334},\\
R_{1213}&=R_{2324}=R_{1424}=R_{3134}=R_{1313}=R_{2424}=R_{1224}\\&=R_{3123}=R_{3114}=R_{4234}=0.
\end{array}
\end{equation}
\end{prop}
\begin{proof}
The local form of \eqref{3.2-P} is given by
\begin{equation}\label{loc-P}
R_{ijab}P_{k}^{a}P_{h}^{b}=R_{ijkh}.\end{equation}
Using \eqref{S} and applying the symmetries of $R$ and the first Bianchi identity to the components of $R$, we obtain \eqref{R-loc-P}.

Conversely, \eqref{loc-P} follows from \eqref{S} and \eqref{R-loc-P}, so \eqref{3.2-P} holds true.
\end{proof}

\begin{prop}\label{the1-P}
If identity \eqref{3.2-P} holds true for a manifold $(M, g, P)$, then the components of the Ricci tensor $\rho$ obey the relations:
\begin{equation}\label{system-rho-P}
\rho_{11}=\rho_{22}=\rho_{33}=\rho_{44},\quad \rho_{13}=\rho_{24},\quad \rho_{12}=\rho_{23}=\rho_{34}=\rho_{14}.
\end{equation}
\end{prop}

\begin{proof}
Due to Proposition~\ref{svoistvo-R-P}, the components of $R$ fulfill \eqref{R-loc-P}.
For brevity, we set
\begin{equation}\label{r1-2}
    \begin{split}
R_{1}=R_{1212}, \quad R_{2}=R_{2114}.
\end{split}
\end{equation}
The inverse matrix of $(g_{ij})$ takes the form
\begin{equation}\label{g-obratna}
    (g^{ik})=\frac{1}{D}\begin{pmatrix}
            \overline{A}& \overline{B}& \overline{C} &\overline{B}\\
            \overline{B}&\overline{A}&\overline{B} &\overline{C}\\
            \overline{C}&\overline{B}&\overline{A}&\overline{B}\\
            \overline{B} & \overline{C} &\overline{B} &\overline{A}\\
            \end{pmatrix},
  \end{equation}
where
\begin{equation}\label{coef-obr}
\begin{split}
\overline{A}=A(A+C)-2B^{2},\quad \overline{B}=B(C-A),\\  \overline{C}=2B^{2}-C(A+C),\quad D=(A-C)\big((A+C)^{2}-4B^{2}\big).
\end{split}
\end{equation}

Having in mind \eqref{def-rho}, \eqref{R-loc-P}, \eqref{r1-2}, and \eqref{g-obratna}, we calculate the components of the Ricci tensor $\rho$ as in:
  \begin{equation}\label{rho1-P}
  \begin{split}
      \rho_{11}&=\rho_{22}=\rho_{33}=\rho_{44}=-\frac{2}{D}(\overline{A}R_{1}+\overline{C}R_{2}),\\
    \rho_{12}&=\rho_{23}=\rho_{34}=\rho_{14}=-\frac{2\overline{B}}{D}(R_{1}+R_{2}),\\
    \rho_{13}&=\rho_{24}=-\frac{2}{D}(\overline{A}R_{2}+\overline{C}R_{1}).
    \end{split}
\end{equation}
Thus, equalities \eqref{system-rho-P} are valid.
\end{proof}
The quantity $\tau^{*}$ associated with the scalar curvature $\tau$ is determined by
\begin{equation}\label{def-rho2}
 \tau^{*}=\tilde{g}^{ij}\rho(e_{i}, e_{j}).
\end{equation}
In \cite{griba-mek2}, it was shown that every four-dimensional Riemannian almost-product manifold satisfying \eqref{3.2-P} is almost-Einstein. Here, we specify and reformulate this result for the particular manifold $(M, g, P)$ along with a short proof applicable to our case.
 \begin{prop}\label{prop5-P}
Any manifold $(M, g, P)$ satisfying \eqref{3.2-P} is an almost-Einstein manifold.
\end{prop}
\begin{proof}
The inverse matrix of $(\tilde{g}_{ij})$ is \begin{equation}\label{tilde-g-obratna}
(\tilde{g}^{ij})=\frac{1}{D}\begin{pmatrix}
       \overline{C}& \overline{B}& \overline{A} &\overline{B}\\
            \overline{B}&\overline{C}&\overline{B} &\overline{A}\\
            \overline{A}&\overline{B}&\overline{C}&\overline{B}\\
            \overline{B} & \overline{A} &\overline{B} &\overline{C}\\
    \end{pmatrix}.
\end{equation}
Using \eqref{def-rho}, \eqref{system-rho-P}, \eqref{g-obratna}, \eqref{rho1-P}, \eqref{def-rho2}, and \eqref{tilde-g-obratna} we calculate the values of the scalar curvature $\tau$ and the associated quantity $\tau^{*}$ as follows:
  \begin{equation*}
  \begin{split}
    \tau =\frac{4}{D}(\overline{A}\rho_{11}+2\overline{B}\rho_{12}+\overline{C}\rho_{13}),\quad
        \tau^{*}=\frac{4}{D}(\overline{C}\rho_{11}+2\overline{B}\rho_{12}+\overline{A}\rho_{13}).
        \end{split}
\end{equation*}
From these equalities and \eqref{coef-obr}, we obtain
 \begin{equation}\label{tau-P}
  \begin{split}
    \tau -\tau^{*}&=\frac{4}{A-C}(\rho_{11}-\rho_{13}),\\
       \tau+ \tau^{*}&=\frac{4}{(A+C)^{2}-4B^{2}}((A+C)(\rho_{11}+\rho_{13})-4B\rho_{12}).
        \end{split}
\end{equation}

On the other hand equalities \eqref{coef-obr} and \eqref{rho1-P} imply the relations for
\begin{equation*}
  R_{1}=-\frac{1}{2(A+C)}(\overline{A}\rho_{11}-\overline{C}\rho_{13}),\quad R_{2}=\frac{1}{2(A+C)}(\overline{C}\rho_{11}-\overline{A}\rho_{13}),
\end{equation*}
leading to
\begin{equation*}
  R_{1}+R_{2}=\frac{\overline{C}-\overline{A}}{2(A+C)}(\rho_{11}+\rho_{13}).
\end{equation*}
From \eqref{coef-obr}, \eqref{tau-P} and the above equality, we obtain successively
\begin{equation*}
     \rho_{12}=\frac{B}{4}(\tau +\tau^{*}),\quad
     \rho_{11}=\frac{1}{4}(\tau A+\tau^{*}C),\quad \rho_{13}=\frac{1}{4}(\tau C+\tau^{*}A).
\end{equation*}
Bearing in mind \eqref{metricg}, \eqref{metricJ}, \eqref{system-rho-P}, and the above equalities, we find
\begin{equation*}
     \rho_{ij}=\frac{\tau}{4}g_{ij}+\frac{\tau^{*}}{4}\tilde{g}_{ij},
\end{equation*}
which yields
\begin{equation}\label{rho-AE-P}
     \rho(x,y)=\frac{\tau}{4}g(x,y)+\frac{\tau^{*}}{4}\tilde{g}(x,y).
\end{equation}
Therefore, comparing \eqref{rho-AE-P} with \eqref{AE}, we conclude that $(M, g, P)$ is an almost-Einstein manifold.
\end{proof}
Immediately it follows the next
\begin{corollary}\label{prop-enstein-P}
A manifold $(M, g, P)$ satisfying \eqref{3.2-P} is Einstein if and only if $\tau^{*}=0$.
\end{corollary}
 \begin{corollary}
Every manifold $(M, g, P)$ under condition \eqref{keler} is an almost-Einstein manifold.
\end{corollary}
In \cite{griba-mek2}, it is proved that
if property \eqref{3.2-P} holds true for a manifold $(M, g, P)$, then the curvature tensor $R$ can be written as
 \begin{equation}\label{fR-P}
R = \frac{\tau}{8}\pi_{1}+\frac{\tau^{*}}{8}\pi_{2},
\end{equation}
where \begin{equation}\label{pi-pi} 
\begin{split}\pi_{1}(x, y, z, u )=g(y, z)g(x, u) &- g(x, z)g(y, u)\\& \mbox{} +\tilde{g}(y, z)\tilde{g}(x, u)
 - \tilde{g}(x, z)\tilde{g}(y, u),\\
\pi_{2}(x, y, z, u)=g(y, z)\tilde{g}(x, u)&-g(x, z)\tilde{g}(y, u)\\& \mbox{} +\tilde{g}(y, z)g(x, u)
 -\tilde{g}(x, z)g(y, u).
 \end{split}
\end{equation}
The above expression is very useful for obtaining the curvatures of the manifolds that we find in Section~\ref{sec:4}.

\section{Almost-Einstein manifolds $(M, g, Q)$}

In this section, we consider a manifold $(M, g, Q)$ whose curvature tensor $R$ preserves the action of $Q$ on the last two arguments. This condition is stronger than \eqref{3.2-P}.

The condition for $Q$ to be parallel with respect to $\nabla$ is given by
\begin{equation*}
\nabla Q=0.\end{equation*}
In \cite{raz}, it was proved that this equality implies for $Q$ and $R$ the following condition:
\begin{equation}\label{3.2} R(x,y,Qz,Qu)=R(x,y,z,u)\end{equation}
We note that \eqref{3.2} yields \eqref{3.2-P}. Then using ideas and calculations of the previous section, we
establish the following propositions.
\begin{prop}\label{svoistvo-R}
Property \eqref{3.2} of the curvature tensor $R$ on the manifold $(M, g, Q)$ is equivalent to the conditions:
\begin{equation}\label{R-loc}
\begin{array}{ll}
R_{1212}&=R_{3434}=R_{2323}=R_{1414}=R_{1234}=R_{4123}\\&=R_{2114}=R_{4134}=R_{1223}=R_{2334}\\
R_{1213}&=R_{2324}=R_{1424}=R_{3134}=R_{1313}=R_{2424}\\&=R_{1224}=R_{3123}=R_{3114}=R_{4234}=0.
\end{array}
\end{equation}
\end{prop}
\begin{proof}
The local form of \eqref{3.2} is
\begin{equation}\label{loc-L3}
R_{ijab}Q_{k}^{a}Q_{h}^{b}=R_{ijkh}.\end{equation}
Using \eqref{Q} and applying the symmetries along with the first Bianchi identity to the components of $R$, we obtain \eqref{R-loc}.

Vice versa, \eqref{loc-L3} follows from \eqref{Q} and \eqref{R-loc}, ensuring that \eqref{3.2} holds true.
\end{proof}

\begin{prop}\label{the1-Q}
If identity \eqref{3.2} holds true for a manifold $(M, g, Q)$, then the components of the Ricci tensor $\rho$ obey the relations:
\begin{equation}\label{system-rho-Q}
\begin{split}
\rho_{11}=\rho_{22}=\rho_{33}=\rho_{44}=-\rho_{13}=-\rho_{24},\quad \rho_{12}=\rho_{23}=\rho_{34}=\rho_{14}=0.
\end{split}
\end{equation}
\end{prop}

\begin{proof}
Due to Proposition~\ref{svoistvo-R}, the components of $R$ comply with \eqref{R-loc}.
For brevity, we set $R_{1}=R_{1212}$.
Having in mind \eqref{def-rho}, \eqref{g-obratna}, and \eqref{R-loc}, we compute the components of the Ricci tensor $\rho$ as detailed in:
  \begin{equation*}
  \begin{split}
      \rho_{11}&=\rho_{22}=\rho_{33}=\rho_{44}=\frac{2R_{1}}{D}(-\overline{A}+\overline{C}),\\
    \rho_{12}&=\rho_{23}=\rho_{34}=-\rho_{14}=0,\\
    \rho_{13}&=\rho_{24}=\frac{2R_{1}}{D}(\overline{A}-\overline{C}).
    \end{split}
\end{equation*}
Thus, the equalities \eqref{system-rho-Q} are valid.
\end{proof}
By extending the previous arguments and applying the specific restrictions of the structure $Q$, we establish the corresponding almost-Einstein property for $(M, g, Q)$.
 \begin{prop}\label{prop5-Q}
Any manifold $(M, g, Q)$ satisfying \eqref{3.2} is an almost-Einstein manifold.
\end{prop}
\begin{proof}
Using \eqref{g-obratna}, \eqref{def-rho2}, \eqref{system-rho-Q}, and \eqref{tilde-g-obratna}, we determine the values of the scalar curvature $\tau$ and the associated quantity $\tau^{*}$ as follows:
  \begin{equation*}\label{tau-Q}
    \tau =\frac{4}{D}(\overline{A}-\overline{C})\rho_{11},\quad
        \tau^{*} =\frac{4}{D}(\overline{C}-\overline{A})\rho_{11}.
\end{equation*}
These equalities
and \eqref{coef-obr} imply
\begin{equation}\label{rho-tau-Q}
     \tau^{*}=-\tau, \quad \rho_{11}=\frac{\tau}{4}(A-C).
\end{equation}
From \eqref{rho-tau-Q}, bearing in mind \eqref{metricg} and \eqref{metricJ}, we get
\begin{equation*}
    \rho_{11}=\frac{\tau}{4}(g_{11}-\tilde{g}_{11}).
\end{equation*}
Taking into account \eqref{metricg}, \eqref{metricJ}, \eqref{system-rho-Q}, and \eqref{rho-tau-Q}, which yields
\begin{equation*}
     \rho_{ij}=\frac{\tau}{4}(g_{ij}-\tilde{g}_{ij}),
\end{equation*}
i.e.
\begin{equation}\label{rho-AE-Q}
     \rho(x,y)=\frac{\tau}{4}(g(x,y)-\tilde{g}(x,y)).
\end{equation}
Therefore, comparing \eqref{rho-AE-Q} with \eqref{AE}, we state that $(M, g, Q)$ is an almost-Einstein manifold.
\end{proof}
Immediately it follows the next
\begin{rem}\label{prop-enstein-Q}
A manifold $(M, g, Q)$ satisfying \eqref{3.2} does not admit Einstein metric.
\end{rem}
As a direct consequence of formula \eqref{fR-P} and the first identity in \eqref{rho-tau-Q}, we establish the explicit geometric representation for the curvature tensor.
\begin{prop}\label{th3.R}
If identity \eqref{3.2} holds true for a manifold $(M, g, Q)$, then the curvature tensor $R$ admits the presentation
 \begin{equation}\label{fR-Q}
R = \frac{\tau}{8}(\pi_{1}-\pi_{2}),
\end{equation}
where $\pi_{1}$ and $\pi_{2}$  are given by \eqref{pi-pi}.
\end{prop}
\section{Curvature properties}\label{sec:4}

In this section, we consider the sectional and Ricci curvatures of $(M, g, Q)$ and $(M, g, P)$ determined by special bases of the tangent space $T_{p}M$.
\begin{defn}
A basis of the form $\{x, Qx, Q^{2}x, Q^{3}x\}$ of $T_{p}M$ is called a $Q$-\textit{basis}. In this case, \textit{the vector $x$ is said to induce a $Q$-basis of} $T_{p}M$.
\end{defn}
In \cite{raz}, it was proved that if a vector $x$ induces a $Q$-basis, the angles between the basis vectors satisfy
\begin{equation*}
\begin{split}
& \angle(x,Qx)=\angle(Qx,Q^{2}x)=\angle(x,Q^{3}x)=\angle(Q^{2}x,Q^{3}x)=\varphi,\\
& \angle(x,Q^{2}x)=\angle(Qx,Q^{3}x)=\phi.
\end{split}
\end{equation*}
Moreover, the angles $\varphi=\angle(x,Qx)$ and $\phi=\angle(x,Q^{2}x)$ obey the following inequalities:
\begin{equation*}
  \frac{\phi}{2}<\varphi<\pi-\frac{\phi}{2}.
\end{equation*}
\begin{theorem}\label{predL1}
Let $(M, g, Q)$ be a manifold satisfying property \eqref{3.2}. If a vector $x$ induces a $Q$-basis, then the sectional curvatures of the basic $2$-planes are given by:
\begin{equation*}
\begin{split}
k(x,Qx)&=k(Qx,Q^{2}x)=k(Q^{2}x,Q^{3}x)=k(Q^{3}x,x)\\&=-\frac{\tau}{8\sin^{2}\varphi}(\cos\phi-1)^{2},\\
k(x,Q^{2}x)&=k(Qx,Q^{3}x)=0,
\end{split}
\end{equation*}
where $\varphi=\angle(x, Qx),$ $\phi=\angle(x, Q^{2}x).$
\end{theorem}
\begin{proof}
If $(M, g, Q)$ satisfies \eqref{3.2}, it follows that
\begin{equation*}
R(Qx,Qy,Qz,Qu)=R(x,y,z,u).\end{equation*}
From this equality,
bearing in mind \eqref{Q4} and the symmetries of $R$, we obtain
\begin{equation*}
\begin{split}
  R(x, Q^{2}x, x, Q^{2}x)&=R(Qx, Q^{3}x, Qx, Q^{3}x)\\R(x, Qx, x, Qx)&=R(x, Q^{3}x, x, Q^{3}x)=
  R(Qx, Q^{2}x, Qx, Q^{2}x)\\&=R(Q^{2}x, Q^{3}x, Q^{2}x, Q^{3}x).
  \end{split}
\end{equation*}
 Using \eqref{isometry} and the angle relations, we calculate:
 \begin{equation}\label{angles-g}
\begin{split}
 g(x,Qx)=g(Qx,Q^{2}x)&=g(Q^{2}x,Q^{3}x)=g(x,Q^{3}x)=g(x,x)\cos\varphi,\\
 g(x,Q^{2}x)=g(Qx,Q^{3}x)&=g(x,x)\cos\phi
\end{split}
\end{equation}
and
\begin{equation}\label{angles-tilde-g}
\begin{split}
\tilde{g}(x,Qx)= \tilde{g}(Qx,Q^{2}x)&= \tilde{g}(Q^{2}x,Q^{3}x)\\&= \tilde{g}(x,Q^{3}x)=g(x,x)\cos\varphi,\\
 \tilde{g}(x,x)=\cos\phi, \quad \tilde{g}(x,Q^{2}x)&= \tilde{g}(Qx,Q^{3}x)=g(x,x).
\end{split}
\end{equation}
 Applying  \eqref{pi-pi}, \eqref{fR-Q}, \eqref{angles-g}, and \eqref{angles-tilde-g} to \eqref{3.3}, we arrive at the result.
\end{proof}

\begin{theorem}\label{t2.5}
Let $(M, g, P)$ be a manifold satisfying \eqref{3.2-P} and let a vector $x$ induce a $Q$-basis. Then, the sectional curvatures of the $2$-planes determined by the $Q$-basis are given by
\begin{equation}\label{obmu}
\begin{split}
    k(x,Qx)&=k(Qx,Q^{2}x)=k(Q^{2}x,Q^{3}x)=k(Q^{3}x,x)\\&=\frac{1}{8\sin^{2}\varphi}\Big(\tau(2\cos^{2}\varphi-\cos^{2}\phi-1)+2\tau^{*}(\cos^{2}\varphi-\cos\phi)\Big),\\
    k(x,Q^{2}x)&=k(Qx,Q^{3}x)=0,
    \end{split}
\end{equation}
where $\varphi=\angle(x, Qx)$, $\phi=\angle(x, Q^{2}x)$.
\end{theorem}
\begin{proof}
By applying \eqref{isometry}, \eqref{fR-P}, \eqref{pi-pi}, \eqref{angles-g}, and \eqref{angles-tilde-g} into \eqref{3.3}, we obtain \eqref{obmu}.
\end{proof}
\begin{corollary}
Let a vector $x$ induce an orthonormal $Q$-basis. Then,
\begin{equation*}
  \begin{split}
    k(x,Qx)=-\frac{\tau}{8}.
    \end{split}
\end{equation*}
\end{corollary}
\begin{proof}
The proof follows directly from \eqref{obmu}, when $\varphi=\phi=\frac{\pi}{2}$.
\end{proof}
Due to Theorem~\ref{t2.5} and Corollary \ref{prop-enstein-P} we establish the following.
\begin{prop} If a manifold $(M, g, P)$ under condition \eqref{3.2-P} is Einstein, then the sectional curvatures of the 2-planes determined by a $Q$-basis are take the form
\begin{equation*}
\begin{split}
    k(x,Px)=\frac{\tau(2\cos^{2}\varphi-\cos^{2}\phi-1)}{8\sin^{2}\varphi}.
    \end{split}
\end{equation*}
\end{prop}

 \begin{theorem}
Let $(M, g, P)$ satisfy identity \eqref{3.2-P}. If a vector $x$ induces a $Q$-basis, then the Ricci curvatures in the directions of the basis vectors are:
\begin{equation}\label{ricc-P}
     r(x)=r(Qx)=r(Q^{2}x)=r(Q^{3}x)=\frac{\tau}{4}+\frac{\tau^{*}}{4}\cos\phi,
\end{equation}
where $\phi=\angle(x, Q^{2}x)$.
\end{theorem}
\begin{proof}
 In the course of the proof of Theorem~\ref{prop5-P}, we find that $\rho$ is given by \eqref{rho-AE-P}. Therefore, using \eqref{isometry}, we obtain
\begin{equation*}
\begin{split}
    \rho(x, x)=\rho(Qx, Qx)=\rho(Q^{2}x, Q^{2}x)&=\rho(Q^{3}x, Q^{3}x)=\frac{\tau}{4}g(x,x)+\frac{\tau^{*}}{4}\tilde{g}(x,x).
    \end{split}
\end{equation*}

Let a vector $x$ induce a $Q$-basis. Then \eqref{isometry}, \eqref{Ricicurv}, \eqref{angles-g}, \eqref{angles-tilde-g}, and the above equality imply
 \eqref{ricc-P}.
\end{proof}

 \begin{theorem}\label{th-ricci}
Let $(M, g, Q)$ be a manifold satisfying property \eqref{3.2}. If a vector $x$ induces a $Q$-basis, then the Ricci curvatures in the direction of the basis vectors are given by:
\begin{equation}\label{ricc-Q}
     r(x)=r(Qx)=r(Q^{2}x)=r(Q^{3}x)=\frac{\tau}{4}(1-\cos\phi),
\end{equation}
where $\phi=\angle(x, Q^{2}x)$.
\end{theorem}
\begin{proof}
 In this case $\rho$ is given by \eqref{rho-AE-Q}. Therefore, using \eqref{isometry} and \eqref{rho-AE-Q}, we obtain
\begin{equation*}
\begin{split}
    \rho(x, x)=\rho(Qx, Qx)=\rho(Q^{2}x, Q^{2}x)=\rho(Q^{3}x, Q^{3}x)=\frac{\tau}{4}(g(x,x)+\tilde{g}(x,x)).
    \end{split}
\end{equation*}
Let a vector $x$ induce a $Q$-basis. Applying equalities \eqref{isometry}, \eqref{angles-g}, \eqref{angles-tilde-g}, and the above equality to \eqref{Ricicurv} we find \eqref{ricc-Q}.
\end{proof}
\begin{rem}
If the $Q$-basis is orthonormal for a manifold with a parallel structure $Q$ on $(M, g, Q)$ (or with a parallel structure $P$ on $(M, g, P)$), then the Ricci curvatures reduce to $r(x)=r(Qx)=r(Q^{2}x)=r(Q^{3}x)=\frac{\tau}{4}$.
\end{rem}

\subsection{Examples of manifolds $(M, g, Q)$}\label{s4.1}
In this section, we construct a non-trivial example of a manifold $(M, g, Q)$.

Let $\mathbb{E}^{10}$ be the ten-dimensional Euclidean space $(\mathbb{R}^{10}, <\cdot,\cdot>)$ equipped with the standard inner product $$<x,y>=\sum_{i=\overline{1,10}} x^{i}y^{i},$$
where $x(x^{i})$ and $y(y^{i})$ are vectors in $\mathbb{R}^{10}$.

We consider a four-dimensional hypersurface $S\in\mathbb{E}^{10}$ given by the para\-metrization 
\begin{equation}\label{surf}
\begin{split}
S:\ r \big(\Phi\sin\alpha, & \Phi\cos\alpha, \Phi\sin\beta, \Phi\cos\beta, \Phi\sin\gamma, \Phi\cos\gamma, u, v, w, t \big), 
\end{split}
\end{equation}
where $\Phi=\exp(u+v+w+t)$, $\alpha=u+w$, $\beta=v+t$, $\gamma=u+v+w+t$. 

The coefficients of the first fundamental form of $S$ are found to be: 
\begin{equation}\label{prw0}
\begin{split}
    A= 5\Phi^{2}+1,\quad B= 4\Phi^{2},\quad C= 5\Phi^{2}.
    \end{split}
\end{equation}
Evidently, the inequalities  $A>C>B>0$ hold true. 
Therefore the manifold $(S, g, Q)$, whose metric $g$ is determined by these coefficients and whose structure $Q$ is given by \eqref{Q}, belongs to the class of manifolds $(M, g, Q)$ under consideration.

Direct verification shows that these functions and their derivatives satisfy equalities in \eqref{w0}, but fail to satisfy \eqref{paralel-f4}. 
Thus, we establish the following statement:
\begin{prop}\label{kt2}
 The manifold $(S, g, P)$ with a structure $P=Q^{2}$ and the metric $g$ defined by the components \eqref{prw0} is a locally product Riemannian manifold.
\end{prop}

\begin{rem} Geometrically, the four-dimensional hypersurface decomposes into a product of two independent two-dimensional surfaces, $S = S_1\times S_2$. The first quadratic form of $S_1$ is given by $$\di s^{2}_{1}=\frac{1}{2}\di ^{2}U+\frac{1}{2}\di ^{2}V,$$ where $U=u-w$ and $V=v-t$. 

The first quadratic form of $S_2$ is $$\di s^{2}_{2}=\left(\frac{1}{2}+5\Phi_{2}\right)\di ^{2}X+8\Phi^{2}\di X\di Y+\left(\frac{1}{2}+5\Phi_{2}\right)\di ^{2}Y,$$ where $X=u+w$, $Y=v+t$. 

The first of these surfaces is Euclidean, whereas the second one is non-trivial, with a non-constant negative sectional curvature $K=\frac{-4\Phi^{2}}{(9\Phi^{2}+\frac{1}{2})^{2}}$.
 \end{rem}

 \begin{rem} To partially visualise the hypersurface \eqref{surf}, we use an orthogonal projection and fix two of the coordinates, for example, let $\alpha=u+w$ and $\beta=w+t$. Then, the symmetrical projection of $S$ onto the coordinate hyperplane $OXYZ$ yields the surface $S_1$ defined by the parametric equations:\begin{equation*}
\begin{split} S_1:\ X=&\exp(\alpha+\beta)\cos\alpha,\ Y=\exp(\alpha+\beta)\cos\beta,\\ Z=&\exp(\alpha+\beta)\cos(\alpha+\beta).\end{split}
\end{equation*}
 Its visual representation is shown in Figure~\ref{fig1}.
 \end{rem}
\begin{rem} Another orthogonal projection of the hypersurface \eqref{surf} is obtained by fixing the coordinate $u$ and a new variable $\alpha$ (which depends on $v$, $w$, and $t$). In this case, the symmetrical projection of $S$ onto the coordinate hyperplane $OXYZ$ results in the surface $S_2$ with the parametric equations:\begin{equation*}
\begin{split} S_2:\ X=\exp(u)\sin\alpha,\ Y=\exp(u)\cos\alpha,\ Z=u.\end{split}
\end{equation*}
 Its graph is illustrated in Figure~\ref{fig2}.
 \end{rem}

\begin{figure}[H]
\centering
\includegraphics[width=0.4\paperwidth]{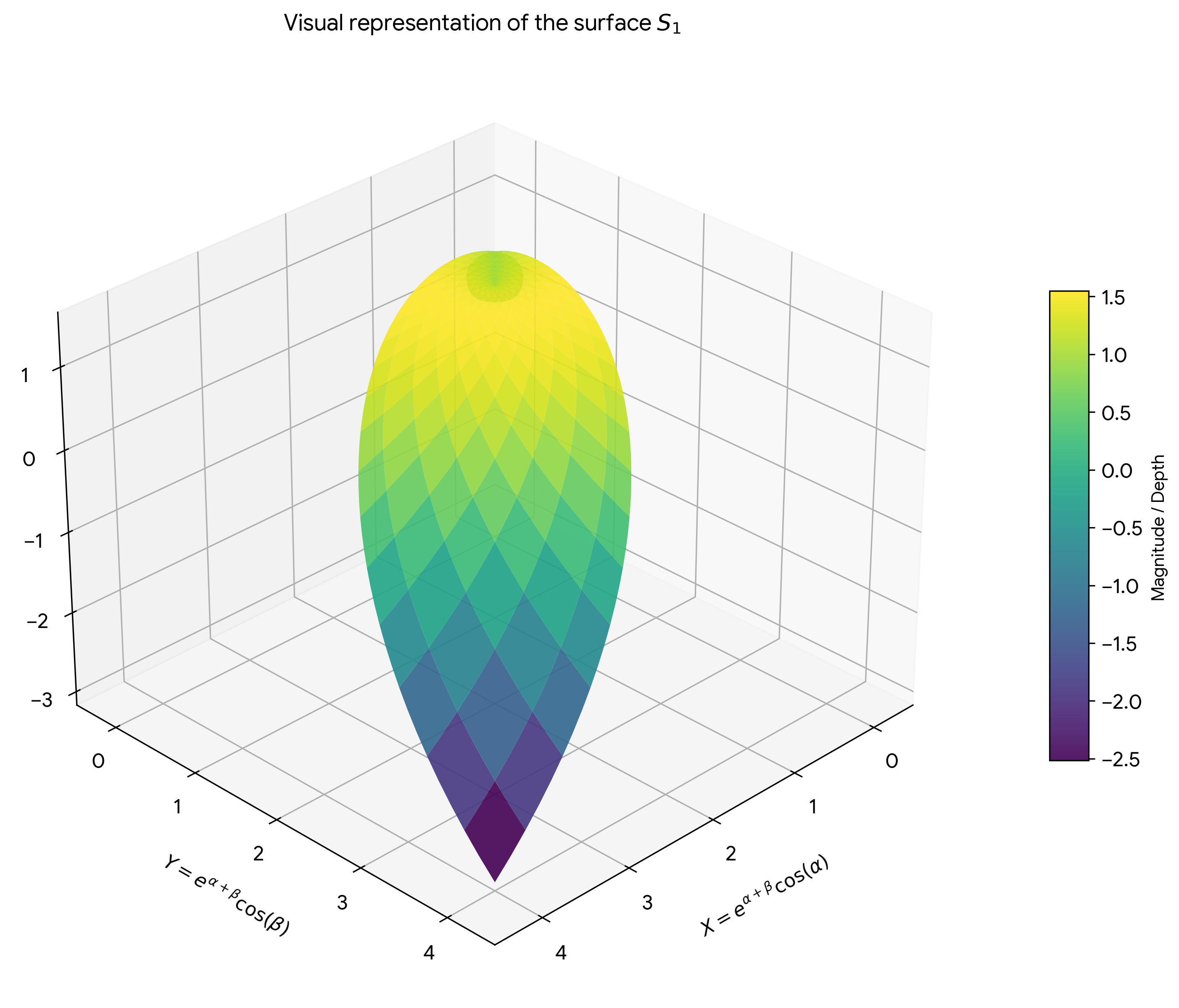}
\caption{Visual representation of the non-trivial surface curvature. The plot was generated using standard Python libraries (matplotlib).} 
\label{fig1}
\end{figure}
\begin{figure}[H]
\includegraphics[width=0.4\paperwidth]{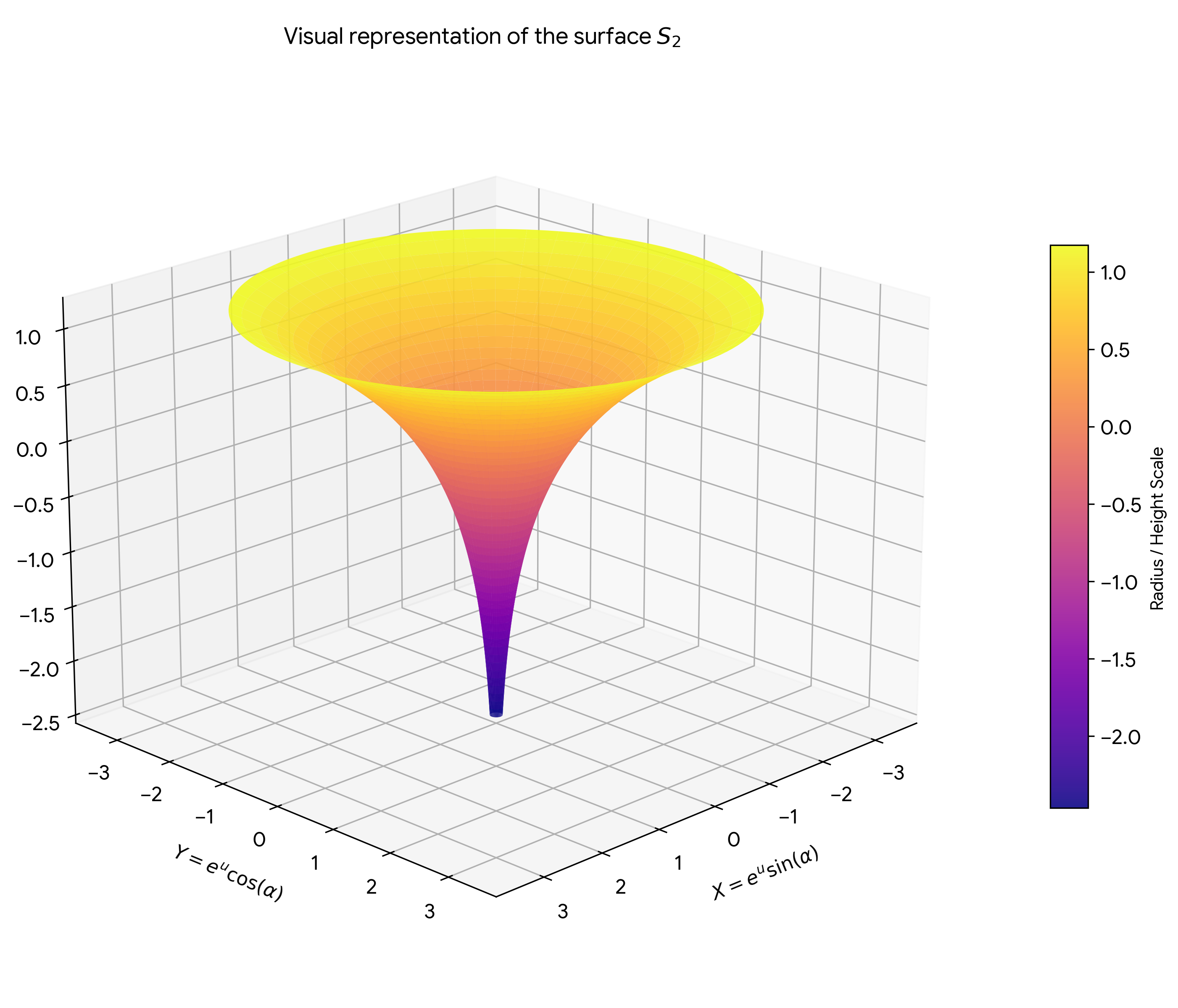}
\caption{Visual representation of the non-trivial surface curvature. The plot was generated using standard Python libraries (matplotlib).} 
\label{fig2}
\end{figure}

\section{Conclusion}
In this work, we have investigated the geometric properties of a structure $Q$ that induces an almost-product structure $P$ on a 4-dimensional Riemannian manifold $(M, g, Q)$. We successfully identified a specific subclass of Riemannian almost-product manifolds characterized by a Riemannian $P$-curvature tensor, and subsequently derived their curvature properties. Additionally, a non-trivial example of a 4-dimensional locally product Riemannian manifold embedded in a 10-dimensional Euclidean space, was constructed.

Future research will focus on finding a non-trivial example of a Riemannian manifold with a parallel structure $Q$.


\begin{thebibliography}{99}

\bibitem{asmari} N. M. AL-Asmari, S. Uddin, M. F. Naghi, \emph{Sequential warped product submanifolds in locally product Riemannian manifolds}, Filomat \textbf{37} (2023), no. 27, 9259--9271. 

\bibitem{dokuzova} I. Dokuzova, \emph{Four-dimensional Riemannian product manifolds with circulant structures}, Stud. Univ. Babe\c{c}-Bolyai Math. \textbf{68} (2023), no. 2, 439--448.

\bibitem{dokuzova1}
I. Dokuzova,  \emph{Curvature properties of $4$-dimensional Riemannian manifolds with a circulant structure}, J. Geom. {\bf 108} (2017), no. 2, 517--527.

 \bibitem{dokuzova2}
I. Dokuzova,  \emph{Almost conformal transformation in a four-dimensional Riemannian manifold with an additional structure}, Int. Electron. J. Geom. {\bf 6} (2013), no. 2, 9--13.


\bibitem{esra}
E. Erkan, K. Takano, M. G\"{u}lbahar \emph{Locally product-like statistical manifolds and their hypersurfaces}, Int. Electron. J. Geom.  \textbf{16} (2023), no. 2,  435--450. 


\bibitem{gribacheva2}
D. Gribacheva, \emph{Curvature properties of two Naveira classes of Riemannian product manifolds}, Plovdiv. Univ. Paisi\v{\i} Khilendarski Nauchn. Trud. Mat. \textbf{39} (2012), no. 3,  31--42. Preprint arXiv:1204.5838


\bibitem{griba-mek}
D.~Gribacheva, D.~Mekerov, \emph{Natural connections on conformal Riemannian $P$-manifolds}, C. R. Acad. Bulgare Sci.  {\bf 65} (2012), no. 5, 581--590.

\bibitem{griba-mek2}
D.~Gribacheva, D.~Mekerov, \emph{Conformal Riemannian $P$-manifolds with connections whose curvature tensors are Riemannian $P$-tensors}, J. Geom.  {\bf 105} (2014), no. 2, 273--286.

\bibitem{jack}
J. C.M. Hughes, J. F. J. Beek, F. V. Kusmartsev, \emph{Warped product Einstein manifolds in four dimensions}, J. Geom. Phys., \textbf{230} (2026), paper no.105987.

\bibitem{nav}
A. M. Naveira,  \emph{A classification of Riemannian almost product manifolds},  Rend. Mat. (7) \textbf{3} (1983), no. 3, 577--592.


\bibitem{pusic}
 N. Pu\v{s}i\'{c}, \emph{On some connections on locally product Riemannian manifolds - Part I}, Novi Sad J. Math. \textbf{41} (2011), no. 2, 29--40.


\bibitem{raz}
D. Razpopov, \emph{Four-dimensional Riemannian manifolds with two circulant structures},  Math. Educ. Math., {\bf 44} (2015), In: Proc. of $44$-th Spring Conf. of UBM, SOK Kamchia, Bulgaria, 179--185.

\bibitem{razpopov2} D. Razpopov, G. Dzhelepov, \emph{Curvature properties of Riemannian manifolds with circulant structures}, Adv. Math. Sci. J. \textbf{9}  (2020), no. 1, 37--47.

\bibitem{sta}
M.~Staikova, K.~Gribachev, \emph{Canonical conections and their conformal invariants on Riemannian $P$-manifolds}, Serdica, vol.~18 (1992), 150--161.

\bibitem{sta2}
M.~Staikova, K.~Gribachev, D. Mekerov, \emph{Riemannian $P$-manifolds of constant sectional curvatures}, Serdica, \textbf{17} (1991), no. 4, 212--219.

\bibitem{yano}
K. Yano, Differential geometry on complex and almost complex spaces. Pure and Applied Math. \textbf{49},  New York, Pergamont Press Book, 326,
(1965).
 

%
\end{thebibliography}
\end{document}